\documentclass[final,leqno,onefignum,onetabnum]{siamltex}

\usepackage{subfigure,upgreek,dsfont,braket,amssymb,amsmath}
\usepackage[small,bf]{caption}
\usepackage{booktabs,dsfont}
\usepackage{algorithm}
\usepackage{float}
\usepackage{graphicx}
\usepackage{algorithmicx}
\usepackage{color,float}

\usepackage{amssymb, amsmath}

\newcommand{\auskommentieren}[1]{}

\newcommand{\beq}{\begin{equation}}
\newcommand{\eeq}{\end{equation}}
\DeclareMathOperator{\graph}{graph}

\newtheorem{remark}[theorem]{Remark}

\title{Optimal control of elliptic surface PDEs with pointwise bounds on the state} 

\author{Ahmad Ahmad Ali \thanks{Fachbereich Mathematik, Universit\"at Hamburg, Bundesstra\ss e 55, 20146 Hamburg, Germany.
{\tt  ahmad.ali@uni-hamburg.de}} \and Michael Hinze \thanks{Fachbereich Mathematik, Universit\"at Hamburg, Bundesstra\ss e 55, 20146 Hamburg, Germany.
{\tt  michael.hinze@uni-hamburg.de}} \and Heiko Kr\"oner \thanks{Fachbereich Mathematik, Universit\"at Hamburg, Bundesstra\ss e 55, 20146 Hamburg, Germany.
{\tt  heiko.kroener@uni-hamburg.de}}}

\begin{document}
\maketitle
\slugger{mms}{xxxx}{xx}{x}{x--x}%slugger should be set to mms, siap, sicomp, sicon, sidma, sima, simax, sinum, siopt, sisc, or sirev

\begin{abstract}
We consider  a
linear-quadratic optimization problem with pointwise bounds on the state for which the constraint is given by the Laplace-Beltrami equation  (to have uniqueness we add an lower order term) on a two-dimensional surface . By using finite elements we approximate the optimization problem by a family of discrete problems 
and prove convergence rates for the discrete controls and the discrete states. Furthermore, assuming (roughly spoken) a higher regularity for the control the order of convergence improves. This extends a result known in an Euclidean setting to the surface case.
\end{abstract}

%\begin{AMS}\end{AMS} 

\section{Introduction}
In applications the situation of a (moving) hypersurface seperating two (moving) regions is
a widespread setting to model various phenomena. In this general setting one may think of 
biological processes
happening in these regions or on the interface between these regions.
Examples for this scenario are
cell membranes seperating the environment from the cell interior, or the interface between the two phases
of a two-phase flow where soluble surfactants in the bulk regions
affect a certain interfacial surfactant concentration, see \cite{GarckeLamStinner2014} and the references therein for 
a two-phase flow example. 

It is a natural to consider optimization problems where the surfactant density on the 
surface plays the role of the
state variable and to assume certain pointwise bounds for the state. 
To address control of the general setting above we consider in our paper a 
linear-quadratic PDE-constrained optimization
problem on a fixed hypersurface (and not phenomena or interactions in or with the regions outside 
the hypersurface). 

The corresponding optimization problem in an Euclidean setting is treated in \cite{HinzeUlbrich} and
we will follow the argumentation therein closely.

There are only few papers which deal with the numerics of linear-quadratic, pde constrained 
optimization problems on surfaces. 
In \cite{HinzeVierling2012} an optimal control problem for the Lapace-Beltrami on surfaces is 
considered and a linear-quadratic parabolic control problem on moving 
surfaces is considered in \cite{Vierling2014} in the case of pointwise box constraints and in \cite{HinzeKroener2016} in the case of pointwise bounds on the state.
%Furthermore, \cite{AliHinzeKroener2016} deals with the numerics of 
%a linear-quadratic optimization problem on surfaces with an elliptic pde as constraint
%and pointwise bounds for the state.

Our paper is organized as follows. 
In Section \ref{section2} we introduce the optimization problem under consideration. Section 
\ref{elementary} contains general material about finite elements on surfaces. Section \ref{section3} states
known $L^{\infty}$-estimates which are the key ingredient in our error estimates. In Section \ref{casas}  we discretize the state equation and in Section \ref{section5} the control problem. Our error estimates are formulated
and proved in Section \ref{section6}. 

\section{The optimization problem}\label{section2}
Let $S$ be a two-dimensional, closed, orientable, embedded surface in $\mathbb{R}^3$ and $(U, (\cdot, \cdot)_U)$ a Hilbert space. We consider the optimization problem
\begin{equation} \label{21}
\begin{cases}
\min_{(y,u)\in Y \times U_{ad}} J(y,u) = \frac{1}{2}\int_S|y-y_0|^2 + \frac{\alpha}{2}\|u-u_0\|_U^2 \\
\text{s.t.} \\
Ay =  Bu  \\
y \in  Y_{ad} = \{y \in L^{\infty}(S): y \le b\}.
\end{cases}
\end{equation}
Here, 
\begin{equation}
A: H^2(S)\rightarrow L^2(S), \quad Ay:= -\Delta_S y+y,
\end{equation}
$\alpha>0$, $y_0 \in H^1(S)$, $u_0\in U$, $b \in L^{\infty}(S)$, $R:U^{*}\rightarrow U$ denotes the inverse of the Frechet-Riesz isomorphism
and 
\begin{equation}
B:U \rightarrow L^2(S) \subset H^1(S)^{*}
\end{equation}
is a linear, continuous operator. We make the assumption that
\begin{equation} \label{10}
\exists_{u \in U_{ad}} \quad G(Bu) < b
\end{equation}
where $G=A^{-1}$, $U_{ad}\subseteq U$ closed and convex and $Y := H^1(S)$.
We have the following theorem.

\begin{theorem} \label{b3}
Let $u \in U_{ad}$ denote the unique solution to (\ref{21}). Then there exists $\mu \in M(S)$ ($M(S)$ denotes the set of Radon measures on $S$) and
$p\in L^2(S)$ so that with $y=G(Bu)$ there holds
\begin{equation} \label{1000}
\int_SpAv = \int_S(y-y_0)v + \int_Sv d\mu\quad \forall v \in H^2(\Omega), 
\end{equation}
\begin{equation} \label{1001}
(RB^{*}p+\alpha(u-u_0), v-u)_U \ge 0 \quad \forall v \in U_{ad}, 
\end{equation}
and
\begin{equation} \label{1002}
\mu\ge 0, \quad y\le b, \quad \int_S(b-y)d\mu =0.
\end{equation}
\end{theorem}
\begin{proof}
The proof of this theorem is along the lines of the the proof of \cite[Theorem 5.2]{casas2} in the Euclidean setting.
\end{proof}

\section{Finite Elements on Surfaces} \label{elementary}
We triangulate $S$ by a family $T_h$ of flat triangles with corners (i.e. nodes) lying on $S$. We denote the surface of class $C^{0,1}$ given by the union of the triangles $\tau \in T_h$ by $S_h$; the union of the corresponding nodes is denoted by $N_h$. Here, $h>0$ denotes a discretization parameter which is related to the triangulation in the following way.
For $\tau \in T$ we define the diameter $\rho(\tau)$  of the smallest disc containing $\tau$, the diameter
 $\sigma(\tau)$ of the largest disc contained in $\tau$ and
\begin{equation}
h = \max_{\tau \in T_h}\rho(\tau), \quad \gamma_h = \min_{\tau \in T_h}\frac{\sigma(\tau)}{h}.
\end{equation}
We assume that the family $(T_h)_{h>0}$ is quasi-uniform, i.e. $\gamma_h \ge \gamma_0 >0$.
We let 
\begin{equation}
X_h = \{v\in C^0(S_h): v_{|\tau}\ \text{linear for all}\ \tau \in T_h \}
\end{equation}
be the space of continuous piecewise linear finite elements.
Let $N$ be a tubular neighborhood of $S$ in which 
the Euclidean metric of $N$ can be written in the coordinates $(x^0, x)=(x^0, x^i)$ of the tubular neighborhood as
\begin{equation}
\bar g_{\alpha \beta} = (dx^0)^2 + \sigma_{ij}(x)dx^idx^j.
\end{equation} 
Here, $x^0$ denotes the globally (in $N$) defined signed distance to $S$ and 
$x=(x^i)_{i=1,2}$ local coordinates for $S$.

For small $h$ we can write $S_h$ as graph (with respect to the coordinates
of the tubular neighborhood) over $S$, i.e.
\begin{equation} \label{30}
S_h = \graph \psi = \{(x^0, x): x^0 = \psi(x), x \in S\}
\end{equation}
where $\psi=\psi_h \in C^{0,1}(S)$ suitable. Note, that 
\begin{equation} \label{31}
|D\psi|_{\sigma}\le c h, \quad |\psi|\le ch^2.
\end{equation}
The induced metric of $S_h$ is given by
\begin{equation}
g_{ij}(\psi(x), x) = \frac{\partial \psi}{\partial x^i}(x) \frac{\partial \psi}{\partial x^j}(x) + \sigma_{ij}(x).
\end{equation}
Hence we have for the metrics, their inverses and their determinants
\begin{equation}
g_{ij}=\sigma_{ij}+O(h^2), \quad 
g^{ij} = \sigma^{ij}+O(h^2) \quad \text{and} \quad
g = \sigma + O(h^2)|\sigma_{ij}\sigma^{ij}|^{\frac{1}{2}}
\end{equation}
where we use summation convention.

 For a function $f:S \rightarrow \mathbb{R}$ we define its lift $\hat f:S_h \rightarrow \mathbb{R}$ to $S_h$ by $f(x) = \hat f(\psi(x), x)$, $x\in S$. For a function $f:S_h \rightarrow \mathbb{R}$ we define its lift $\tilde f:S \rightarrow \mathbb{R}$ to $S$ by $f = \hat{\tilde f}$. This terminus can be obviously extended to subsets.
Let $f \in W^{1,p}(S)$, $g \in W^{1,p^{*}}(S)$, $1\le p \le \infty$ and $p^{*}$
H\"older conjugate of $p$. 
In local coordinates $x=(x^i)$ of $S$ hold
\begin{equation} \label{4}
\int_S \left<D f, D g\right> = \int_S \frac{\partial f}{\partial x^i}\frac{\partial g}{\partial x^j}\sigma^{ij}(x)\sqrt{\sigma(x)}dx^idx^j,
\end{equation}
\begin{equation} \label{5}
\int_{S_h} \left<D \hat f, D \hat  g\right> = \int_{S} \frac{\partial f}
{\partial x^i}\frac{\partial g}{\partial x^j}g^{ij}(\psi(x), x)\sqrt{g(\psi(x), x)}
dx^idx^j,
\end{equation}
\begin{equation}  \label{101}
\int_S \left<D f, D g\right> = \int_{S_h} \left<D \hat f, D \hat  g\right> + 
O(h^2)\|f\|_{W^{1,p}(S)}
\|g\|_{W^{1,p^{*}}(S)},
\end{equation}
and similarly, 
\begin{equation} \label{100}
\int_S f =   \int_{S_h}\hat f+ O(h^2)\|f\|_{L^1(S)}
\end{equation}
where now $f\in L^1(S)$ is sufficient.

The bracket $\left<u,v\right>$ denotes here the scalar product of two tangent vectors $u,v$ (or their covariant counterparts). $\|\cdot \|_{W^{k,p}}$ denotes the usual Sobolev norm, $|\cdot |_{W^{k,p}}=\sum_{|\alpha|=k}\|D^{\alpha}\cdot \|_{L^p}$ and $H^k=W^{k,2}$.

\section{Some $L^{\infty}$-estimates for FE approximations} \label{section3}

We define 
\begin{equation} \label{8}
a:W^{1,p}(S)\times W^{1,p^{*}}(S)\rightarrow \mathbb{R}, \quad a(u,v) =\int_S \left<Du, Dv\right>+uvdx, 
\end{equation}
\begin{equation} \label{9}
a_h:W^{1,p}(S_h)\times W^{1,p^{*}}(S_h)\rightarrow \mathbb{R}, \quad a(u_h,v_h) =\int_{S_h} \left<Du_h, Dv_h\right>
+u_h v_hdx, 
\end{equation}
a discrete operator $G_h:  L^2(S)\rightarrow X_h, v \mapsto G_hv=z_h$ via
\begin{equation} \label{a30}
a_h(z_h, \varphi_h) = \int_{S_h}\hat v\varphi_h \quad \forall \varphi_h \in X_h
\end{equation}
and have the following Lemma.

\begin{lemma} \label{b50}
Let $v\in L^2(S)$ and $z=Gv$, $z_h=G_hv$.

(i) There holds  
\begin{equation} \label{20}
\|z-\tilde z_h\|_{L^{\infty}(S)} \le c h\|v\|.
\end{equation}

(ii) If $v \in W^{1,s}(S)$ for some $1<s<2$ then
\begin{equation}
\|z-\tilde z_h\|_{L^{\infty}(S)} \le c h^{3-\frac{2}{s}} |\log h|\|v\|_{W^{1,s}(S)}.
\end{equation}

(iii) If $v \in L^{\infty}(S)$ then
\begin{equation}
\|z-\tilde z_h\|_{L^{\infty}(S)}  \le c h^2 |\log h|^2\|v\|_{L^{\infty}(S)}.
\end{equation}
\end{lemma}
\begin{proof}
The proof of (i) is as in the Euclidean case and uses \cite{Dziuk1988}.

For $\varphi_h\in X_h$ we define
\begin{equation}
\begin{aligned}
F(\tilde \varphi_h) := a(\tilde z_h-z, \tilde \varphi_h)
\end{aligned}
\end{equation}
and estimate
\begin{equation}
\begin{aligned}
F(\tilde \varphi_h) &a_h(z_h, \varphi_h)+O(h^2)\|z-\tilde z_h\|_{W^{1,2}(S)} 
\|\tilde \varphi_h\|_{W^{1,2}(S)} \\
\le&a_h(z_h, \varphi_h)+O(h)\|z-\tilde z_h\|_{W^{1,2}(S)} 
\|\tilde \varphi_h\|_{W^{1,1}(S)}
\end{aligned}
\end{equation}
where we used an inverse estimate.
Hence $F$ extends by Hahn-Banach theorem to an element in $W^{-1,1}(S)$ with norm of order $O(h^2)\|f\|_{L^2(S)}$ and then by a further application of the Hahn-Banach Theorem to an element in $W^{-2,1}(S)$ with
norm of order $O(h^2)\|f\|_{L^2(S)}$. A careful view shows that we are in the situation of \cite[Theorem 1.2]{schatz1998} if $u\in W^{1, \infty}(S)$.  Hence in this case we have
\begin{equation} \label{1010}
\begin{aligned}
\|z-&\tilde z_h\|_{L^{\infty}(S)}  
 \le c
\left(h | \log h|   \inf_{\chi \in X_h} \| \nabla_{\Gamma} (z-\tilde \chi) \|_{W^{1,\infty}(S)} +h^2\|v\|_{L^2(S)} \right).
\end{aligned}
\end{equation}
We remark that estimate (\ref{1010}) is proved in \cite[Theorem 3.2]{demlow}.

Elliptic regularity theory and standard embedding 
theorems imply $z\in W^{3,s}(S)\subset W^{2, q}(S)$, $q= \frac{2s}{2-s}$,
and hence 
\begin{equation}
\|z\|_{W^{2,q}(S)} \le c \|z\|_{W^{3,s}(S)} \le c \|v\|_{W^{1,s}(S)}.
\end{equation}
From (\ref{1010}) and a well-known interpolation estimate we conclude
\begin{equation} \label{1011}
\|z-\tilde z_h\|_{L^{\infty}(S)} \le c h^{2-\frac{2}{q}} |\log h|\|z\|_{W^{2,q}(S)} +ch^2\|v\|\le c h^{3-\frac{2}{s}}|\log h| \|v\|_{W^{1,s}(S)}
\end{equation}
in view of the relation between $s$ and $q$. This proves (ii).

From elliptic regularity theory we know that $z \in W^{2,q}(S)$ for all $1 \le q < \infty$ with
\begin{equation}
\|z\|_{W^{2,q}(S)} \le C q \|v\|_{L^q(S)} \le c q \|v\|_{L^{\infty}(S)}
\end{equation}
where the constant $C$ is independent from $q$. Combining this with the first inequality in (\ref{1011}) gives
\begin{equation} 
\begin{aligned}
\|z-\tilde z_h\|_{L^{\infty}(S)} 
\le c q h^{2-\frac{2}{q}}|\log h|\|v\|_{L^{\infty}(S)}
\end{aligned}
\end{equation}
so that choosing $q = |\log h|$ proves (iii).
\end{proof}

\section{Finite Element Discretization of $A$} \label{casas}
In this section we adapt the argumentation from \cite{casas} to the surface case.
Let $\mu$ be a regular Borel measure in $S$ we consider the following problem
\begin{equation} \label{a2}
Au = -\Delta_Su +u = \mu.
\end{equation}
Here, $u\in L^2(S)$ is a solution of (\ref{a2}) if
\begin{equation} \label{a11}
\int_SuAv dx = \int_Svd\mu \quad \forall v \in H^2(S).
\end{equation}
Note, that $A$ is self-adjoint.
\begin{theorem} \label{a3}
Let $s\in (1,2)$ and $\mu \in M(S)$. Then there exists a unique solution $u\in W^{1,s}(S)$ of (\ref{a2}) and there holds
\begin{equation} \label{a8}
\|u\|_{W^{1,s}(S)} \le c(s)\|\mu\|_{M(S)}.
\end{equation}
\end{theorem}
\begin{proof}
Let $T: L^2(S)\rightarrow C^0(S)$ be defined by
\begin{equation}
A(Tf) = f, \quad f \in L^2(S). 
\end{equation}
$T$ is well defined in view of $H^2(S)\subset C^0(S)$, linear and continuous. We denote its adjoint operator by $T^{*}\in L(M(S), L^2(S))$. Then we have for all $f\in L^2(S)$ that
\begin{equation} \label{a4}
\int_Sf(T^{*}\mu)dx = \int_S Tf d\mu
\end{equation}
which implies
\begin{equation} \label{a5}
\int_S (T^{*}\mu)Av dx = \int_Svd\mu \quad \forall v\in H^2(S)
\end{equation}
by inserting $f=Av$ in (\ref{a4}). Hence $u=T^{*}\mu$ solves (\ref{a2}). The uniqueness of the solution is obvious. To prove the regularity of $u$ we let $\psi\in C^0(S)$ and $v\in H^2(S)$ be the solution of 
\begin{equation}
Av = \psi.
\end{equation}
From (\ref{a5}) we get
\begin{equation} \label{a6}
\left|\int_S u \psi dx\right| = \left|\int_SuAvdx\right| = \left|\int_Svd\mu\right| \le \|\mu\|_{M(S)} \|v\|_{C^0(S)}.
\end{equation}
By using \cite[Theorem 1.4, p. 319]{Necas} we deduce the existence of $c>0$ so that
\begin{equation} \label{a7}
\|v\|_{C^0(S)} \le c \|\psi\|_{W^{-1,t}(S)}
\end{equation}
where $t>2$ is arbitrary and $c$ depends only on $t$, $S$. 

From (\ref{a6}) and (\ref{a7}) we derive
\begin{equation}
\left|\int_S \psi u dx\right| \le c \|\mu\|_{M(S)} \|\psi\|_{W^{-1,t}(S)}.
\end{equation}
Since $C^0(S)$ is dense in $W^{-1,t}(S)$,  $\frac{1}{s}+\frac{1}{t}=1,$  we conclude that $u \in W^{1,s}(S)$ and (\ref{a8}).
\end{proof}

Let $s \in (1,2)$, $s^{*}$ its H\"older conjugate and consider the bilinear form $a$ in case $p=s$. We consider the following variational problem.  
\begin{equation} \label{a9}
\text{Find } u \in W^{1,s}(S) \text{ so that } a(u,v) = \int_S v d\mu \quad \forall v \in W^{1. s^{*}}(S).
\end{equation}
Note, that in view of $s<2$ we have $s^{*}>2$ so that $W^{1, s^{*}}(S)\subset C^0(S)$.
\begin{theorem}
Problem (\ref{a9}) has a unique solution $u$ and $u$ solves (\ref{a2}).
\end{theorem}
\begin{proof}
Let $u$ be the solution of (\ref{a2}). We show that $u$ is a solution of (\ref{a9}). From Theorem \ref{a3} we know $u \in W^{1,s}_0(S)$ and from (\ref{a11}) we deduce that
\begin{equation} \label{a10}
\int_S v d\mu = a(u,v) \quad \forall v \in H^2(S).
\end{equation}
Hence $u$ solves (\ref{a2}) since $H^2(S)$ is dense in $W^{1, s^{*}}(S)$.

If $u$ solves (\ref{a9}) then (\ref{a10}) holds and implies (\ref{a2}).
\end{proof}

 Let $\mu\in M(S)$ then 
\begin{equation}
C^0(S_h) \ni u \mapsto \int_S \tilde u d\mu 
\end{equation}
is in $(C^0(S_h))^{*}$, positive and via Riesz representation theorem equal to a $\hat \mu \in M(S_h)$.

The discretization of (\ref{a11}) is given by the following problem.
\begin{equation} \label{a12}
\text{Find } u_h \in X_h \text{ so that } a_h(u_h,v_h) = \int_{S_h} v_h d\mu_h \quad \forall v_h \in X_h
\end{equation}
where $\mu_h \in B_{ch}(\hat \mu) \subset M(S_h)$ arbitrary but now fixed. Existence of a solution of (\ref{a12}) follows from uniqueness.
\begin{remark}
If $\mu \in L^2(S)$ then the discretizations (\ref{a12}) and (\ref{a30}) agree for suitable $\mu_h \in B_{ch}(\hat \mu) \subset M(S_h)$.
\end{remark}
\begin{proof}
Let $\mu=f\in L^2(\Omega)$. The map
\begin{equation}
L^2(S_h)\ni v \mapsto \int_{S_h}\hat f v dx -\int_{S_h}vd\hat \mu = \int_{S_h}\hat f v dx -\int_S \tilde v f dx 
\end{equation}
is in $L^2(S_h)$ with norm less or equal $ch^2$ in view of Section \ref{elementary}.
\end{proof}
\begin{lemma} \label{a14}
Let $v\in H^2(S)$ and $v_h \in X_h$ the unique solution of
\begin{equation}
a_h(w_h, v_h) = a(\tilde w_h, v) \quad \forall w_h \in X_h
\end{equation}
then
\begin{equation}\label{a16}
\|v-\tilde v_h\|_{L^{\infty}(S)} \le c h \|v\|_{H^2(S)}.
\end{equation}
\end{lemma}
\begin{proof}
Let $f=Av$ then we have  in view of Section \ref{elementary} that
\begin{equation}
a(\tilde w_h, v) = \int_S\tilde w_h fdx = \int_{S_h}w_h \hat fdx+O(h^2)\|w_h\|_{L^2}\|f\|_{L^2(S)} = \int_{S_h}w_hF
\end{equation}
where $F \in L^2(S_h)$ suitable and $\|\tilde F- f\|_{L^2(S)}\le O(h^2)\|f\|_{L^2(S)}$.
The claim follows as in the Euclidean setting by using the $L^2$-estimate from \cite{Dziuk1988}. 
\end{proof}
\begin{theorem} \label{b10}
Let $u$ be the solution of (\ref{a2}) and $u_h$ the solution of (\ref{a12}). Then
\begin{equation}
\|u-\tilde u_h\|_{L^2(S)} \le c (h\|\mu\|_{M(S)}+\|\hat \mu-\mu_h\|_{M(S_h)}).
\end{equation}
\end{theorem}
\begin{proof}
Let $p \in L^2(S)$ arbitrary and  $v\in H^2(S)$ with 
\begin{equation}
Av = p.
\end{equation}
There holds
\begin{equation}
\begin{aligned}
\int_S(u-\tilde u_h)pdx =&\int_S(u-\tilde u_h)Avdx \\
=& a(u-\tilde u_h, v) \\
=& \int_S v d\mu -a(\tilde u_h, v) \\
=& \int_S v d\mu -a_h(u_h, v_h) \\
\le & \int_S v d\mu - \int_{S_h}v_h d\hat \mu+\|\hat \mu-\mu_h\|_{M(S_h)}\|v_h\|_{C^0(S_h)} \\
=&   \int_S v d\mu - \int_S\tilde v_h d\mu +\|\hat \mu-\mu_h\|_{M(S_h)}\|v_h\|_{C^0(S_h)}\\
\le& \|v-\tilde v_h\|_{C^0(S)}\|\mu\|_{M(S)}+\|\hat \mu-\mu_h\|_{M(S_h)}\|v_h\|_{C^0(S_h)} \\
\le& c(h\|\mu\|_{M(S)}+\|\hat \mu-\mu_h\|_{M(S_h)})\|p\|_{L^2(S)}
\end{aligned}
\end{equation}
where $v_h$ as in Lemma \ref{a14} and we used (\ref{a16}).
\end{proof}

\section{Finite Element Discretization of the optimization problem} \label{section5}
In order to approximate problem (\ref{21}) we  consider the following family of control problems depending on the mesh parameter $h>0$
\begin{equation} \label{2}
\min_{u \in U_{ad}} J_h(u) := \frac{1}{2}\int_{S_h}|y_h-\hat y_0|^2 + \frac{\alpha}{2}\|u-u_{0,h}\|_U^2
\end{equation}
subject to 
\begin{equation} \label{b1}
y_h = G_h(Bu) \quad \wedge \quad y_h(x_j) \le b(x_j), \quad j=1, ..., m.
\end{equation}
Here, $u_{0,h}$ denotes an approximation to $u_0$ with 
\begin{equation}
\|u_0-u_{0,h}\|_{U} \le ch.
\end{equation}
For every $h>0$ the optimization problem  (\ref{2}),(\ref{b1}) agrees with the problem which is stated in \cite[(3.59)]{HinzeUlbrich} apart from the fact that our problem is defined on $S_h$ and the problem stated in \cite[(3.59)]{HinzeUlbrich} is defined in an open and bounded subset $\Omega\subset \mathbb{R}^2$. This difference does not effect the procedure how existence of an optimal solution and necessary optimality conditions are derived. Hence we get using \cite[Lemma 3.2]{HinzeUlbrich} and the definition
\begin{equation}
\hat Bu = \widehat{Bu}\in L^2(S_h), \quad u \in U, 
\end{equation}
that the following Lemma holds.
\begin{lemma} \label{1012}
Problem (\ref{2}) has a unique solution $u_h \in U_{ad}$. There exist $\mu_1, ..., \mu_m \in \mathbb{R}$
and $p_h \in X_h$ so that with $y_h=G_h(Bu_h)$
we have
\begin{equation} \label{23}
\begin{aligned}
a_h(v_h, p_h) = \int_{S_h}(y_h-\hat y_0)v_h +\sum_{j=1}^m\mu_jv_h(x_j) \quad \forall v_h \in X_h \\
(R{\hat B}^{*}p_h + \alpha(u_h-u_{0,h}), v-u_h)_U \ge 0 \quad \forall v \in U_{ad}, \\
\mu_j \ge 0, \quad y_h(x_j) \le b(x_j), \quad j=1, ..., m, \text{ and }
\sum_{j=1}^m\mu_j(b(x_j)-y_h(x_j)) =0.
\end{aligned}
\end{equation}
\end{lemma}

We prove the following a priori bounds which are uniform in $h$.
\begin{lemma} \label{c5}
Let $u_h$, $\mu_j$,  $p_h$ and $y_h$ as in the previous Lemma \ref{1012}. Setting $\mu_h = \sum_{j=1}^m \mu_j\delta_{x_j}$  by abusing notation there exists $\bar h>0$ so that
\begin{equation}
\|y_h\|+ \|u_h\|_U+ \|\mu_h\|_{M(S_h)} \le C \quad \text{ for all } 0<h \le \bar h. 
\end{equation}
\end{lemma}
\begin{proof}
Let $\tilde u$ denote an element satisfying (\ref{10}). Since $G(B\tilde u)$
is continuous there exists $\delta>0$ so that
\begin{equation}
G(B\tilde u)\le b-\delta \quad \text{in } S. 
\end{equation}
From (\ref{20}) we deduce that there is $h_0>0$ so that for all $0<h\le h_0$
\begin{equation} \label{22}
G_h(B\tilde u) \le \hat b \quad \forall\ 0<h\le h_0
\end{equation}
so that
\begin{equation} 
J_h(u_h) \le J_h(\tilde u)  \quad \forall\ 0<h\le h_0
\end{equation}
and hence 
\begin{equation}\label{24}
\|u_h\|_U, \|y_h\|\le c \quad \forall\ 0<h\le h_0.
\end{equation}
Let $u$ denote the unique solution of (\ref{21}), cf. Theorem \ref{b3}. From (\ref{22}) and (\ref{20}) we infer that $v:=\frac{1}{2}u+\frac{1}{2}\tilde u$ satisfies
\begin{equation} \label{25}
\begin{aligned}
\widetilde{G_h(Bv)} \le& \frac{1}{2}G(Bu) + \frac{1}{2}G(B\tilde u)+ch(\|Bu\|+\|B\tilde u\|) \\
\le& b-\frac{\delta}{2} + ch(\|u\|_U+\|\tilde u\|_U) \\
\le& b-\frac{\delta}{4}
\end{aligned}
\end{equation}
for $0<h\le \bar h$ with $0<\bar h \le h_0$ suitable.

Since $v \in U_{ad}$ properties (\ref{23}), (\ref{24}) and (\ref{25}) imply
\begin{equation}
\begin{aligned}
0 \le& (R\hat B^{*}p_h+\alpha (u_h-u_{0,h}), v-u_h)_U \\
=& \int_{S_h}\hat B(v-u_h)p_h + \alpha(u_h-u_{0,h}, v-u_h)_U \\
=& a_h(G_h(Bv)-y_h, p_h)+\alpha(u_h-u_{0,h}, v-u_h)_U \\
=& \int_{S_h}(G_h(Bv)-y_h)(y_h-\hat y_0)+\sum_{j=1}^m\mu_j (G_h(Bv)-y_h)(x_j)\\
&+\alpha(u_h-u_{0,h}, v-u_h)_U \\
\le& C + \sum_{j=1}^m\mu_j\left(b(x_j)-\frac{\delta}{4}-y_h(x_j)\right) \\
=& C-\frac{\delta}{4}\sum_{j=1}^m\mu_j
\end{aligned}
\end{equation}
where the last equality follows from (\ref{23}). We conclude
\begin{equation}
\|\mu_h\|_{M(S)} \le c
\end{equation}
and the lemma is proved.
\end{proof}

\section{Error estimates} \label{section6}
In the following we assume that $\mu_h\in B_{ch^2}(\hat \mu)\subset M(S_h)$ and state the following theorem.
\begin{theorem} \label{main_result}
 Let $u$ and $u_h$ be the solutions of (\ref{21}) and (\ref{2}) respectively. Then 
\begin{equation}
\|u-u_h\|_U + \|y-\tilde y_h\|_{H^1(S)} \le ch^{\frac{1}{2}}.
\end{equation}
If in addition $Bu \in W^{1,s}(S)$ for some $s\in (1, 2)$ then
\begin{equation}
\|u-u_h\|_U + \|y-\tilde y_h\|_{H^1(S)} \le c h^{\frac{3}{2}-\frac{1}{s}}\sqrt{|\log h|}.
\end{equation}
\end{theorem}
\begin{proof}
We test (\ref{23}) with $u_h$ and (\ref{b3}) with $u$. Adding the resulting inequalities gives
\begin{equation}
(R(B^{*}p-\hat B^{*}p_h)-\alpha(u_0-u_{0,h})+\alpha(u-u_h), u_h-u)_U \ge 0.
\end{equation}
We recall the lift operator
\begin{equation}
L^2(S)\rightarrow L^2(S_h), \quad u \mapsto \hat u
\end{equation}
and introduce its adjoint 
\begin{equation}
L^2(S_h)\rightarrow L^2(S), \quad u \mapsto \check u
\end{equation}
which is $O(h^2)$ close to
\begin{equation}
L^2(S_h)\rightarrow L^2(S), \quad u \mapsto \tilde u.
\end{equation}
There holds $\hat B^{*}p_h=B^{*}\check{p_h}$ so that we conclude
\begin{equation} \label{b11}
\alpha \|u-u_h\|_U^2 \le \int_SB(u_h-u)(p-\check{p_h})-\alpha(u_0-u_{0,h}, u_h-u)_U.
\end{equation}
Let $y^h = G_h(Bu)\in X_h$ and denote by $p^h\in X_h$ the unique solution of
\begin{equation}
a_h(w_h, p^h) = \int_{S_h}(\hat y-\hat y_0)w_h + \int_{S_h} w_h d\mu_h \quad \forall Êw_h \in X_h.
\end{equation}
Applying Theorem \ref{b10} with $\tilde \mu = (y-y_0)+\mu$ we infer
\begin{equation}
\|p-\widetilde{p^h}\|_{L^2(S)} \le ch(\|y-y_0\|_{L^2(S)}+\|\mu\|_{M(S)}+\|\hat \mu-\mu_h\|_{M(S)}).
\end{equation}
We rewrite the first term on the right-hand side of (\ref{b11})
\begin{equation} \label{b20}
\begin{aligned}
\int_SB(u_h-u)&  (p-\check p_h) 
\\& = \int_SB(u_h-u)(p-\widetilde {p^h}) + \int_SB(u_h-u)(\widetilde{p^h}-\check p_h) \\
=& \int_SB(u_h-u)(p-\widetilde {p^h}) + \int_{S_h}\widehat{B(u_h-u)}(p^h- p_h)\\
&+ O(h^2)\|u-u_h\|_U\|\widetilde{p^h}-\check p_h\|_{L^2(S)}+I_1 \\
=& \int_SB(u_h-u)(p-\widetilde {p^h}) + a_h(y_h-y^h,p^h- p_h)\\
&+ O(h^2)\|u-u_h\|_U\|\widetilde{p^h}-\check p_h\|_{L^2(S)}+I_1 \\
=& \int_SB(u_h-u)(p-\widetilde{p^h})+\int_{S_h}(\hat y-y_h)(y_h-y^h) \\
&+ \int_{S_h}y_h-y^hd\mu_h - \sum_{j=1}^m\mu_j(y_h-y^h)(x_j)+I_1 \\
=& \int_SB(u_h-u)(p-\widetilde{p^h})-\|\hat y-y_h\|^2_{L^2(S_h)} \\
&+\int_{S_h}(\hat y-y_h)(\hat y-y^h) + \int_{S_h}y_h-y^hd\mu_h + \sum_{j=1}^m\mu_j(y^h-y_h)(x_j)\\
&+I_1 \end{aligned}
\end{equation}
where
\begin{equation}
I_1 = \int_{S_h}\widehat{B(u_h-u)}(p_h-\hat{ \check p}_h)
\end{equation}
and
\begin{equation}
|I_1| \le O(h^2)\|p_h\|_{L^2(S_h)}\|u-u_h\|_U
\end{equation}
After inserting (\ref{b20}) into (\ref{b11}) and using Young's inequality we obtain in view of (3.71), (3.55) and (3.60)
\begin{equation} \label{b40}
\begin{aligned}
\frac{\alpha}{2}&\|u-u_h\|_U^2+\frac{1}{2}\|\hat y-y_h\|^2 \\
&\le c(\|p-\widetilde{p^h}\|^2_{L^2(S)}+\|\hat y-y^h\|^2_{L^2(S_h)}+\|u_0-u_{0,h}\|^2_U) +
 \int_{S_h}(y_h-y^h)d\mu_h \\
&+ \sum_{j=1}^m\mu_j(y^h-y_h)(x_j) +|I_1|.
\end{aligned}
\end{equation}
We have
\begin{equation}
y_h-y^h \le I_hb-\hat b+\hat b-\hat y+\hat y-y^h
\end{equation}
and hence
\begin{equation}
\begin{aligned}
\int_{S_h}y_h-y^h d \mu_h \le& \|\mu_h\|_{M(S_h)}(\|I_h b -\hat b\|_{L^{\infty}(S_h)}
+\|\hat y-y^h\|_{L^{\infty}(S_h)}) \\
&+ O(h^2)\|\hat b-\hat y\|_{L^{\infty}(S_h)} + \int_Sb-y d\mu
\end{aligned}
\end{equation}
where the integral on the right-hand side is less or equal zero. Furthermore, we have
\begin{equation} \label{c4}
\begin{aligned}
\sum_{j=1}^m\mu_j(y^h-y_h)(x_j) =& \sum_{j=1}^m\mu_j(y^h-y+y-b+b-y_h)(x_j) \\
\le& \|y^h-\hat y\|_{L^{\infty}(S_h)} \sum_{j=1}^m\mu_j
\end{aligned}
\end{equation}
where we used $y \le b$ and $\sum_{j=1}^m\mu_j(b-y_h)(x_j)=0$.

Using these estimates we can bound the right-hand side of (\ref{b40}) from above by
\begin{equation} \label{41}
\begin{aligned}
ch^2(1&+\|y-y_0\|^2_{L^2(S)}+\|\mu\|_{M(S)}^2+\|u\|_{L^2(S)}^2)\\
&+O(h^2)\|p_h\|_{L^2(S_h)}\|u-u_h\|_U
+\|y^h-y\|_{L^{\infty}(S_h)}. 
\end{aligned}
\end{equation}

Testing (\ref{23}) with $p_h$ yields
\begin{equation} \label{c6}
\|p_h\|^2_{L^2(S_h)} \le c \|p_h\|_{L^2(S_h)} + \|p_h\|_{L^{\infty}(S_h)} \le c h^{-1}\|p_h\|_{L^2(S_h)}
\end{equation}
where we used for the last inequality an inverse inequality.
We conclude
\begin{equation}
\|p_h\|_{L^2(S_h)}\le c h^{-1}.
\end{equation}
Putting facts together shows that the right-hand side of (\ref{b40}) can be bounded from above by 
\begin{equation} \label{60}
ch^2 + \|y^h-\hat y\|_{L^{\infty}(S_h)}.
\end{equation}
The norm in (\ref{60}) can be estimated by $ch\|u\|_{L^2(S)}$ by using (\ref{20}) or by 
\begin{equation}
ch^{3-\frac{2}{s}}|\log h|\|u\|_{L^2(S)}
\end{equation} 
by using Lemma \ref{b50} depending on the assumption on $Bu$.
\end{proof}

\begin{corollary}
Let $u$ and $u_h$ be as in Theorem \ref{main_result} (i) and assume that $Bu, Bu_h\in L^{\infty}(S)$ are uniformly bounded in the $L^{\infty}$-norm.
Then, for $h$ small enough
\begin{equation}
\|u-u_h\|_U+\|y-y_h\|_{H^1} \le c h |\log h|.
\end{equation} 
\end{corollary}

\begin{proof}
We set $\bar y= GBu_h$ and rewrite the first summand on the right-hand side of (\ref{b11}) as follows
\begin{equation} \label{1013}
\begin{aligned}
\int_S B(u_h-u)(p-\check p_h) =& \int_S pA(\bar y-y)-\int_{S_h}\widehat{B(u_h-u)}\hat{\check p}_h
+O(h^2)\|u_h-u\|_U\|\check p_h\|_{L^2(S)} \\
=& \int_S pA(\bar y-y)- a_h(y_h-y^h, p_h)
+O(h^2)\|u_h-u\|_U \| {\check p}_h\|_{L^2(S)} \\
\stackrel{(\ref{1000}), (\ref{23})}{=}&
 \int_S(y-y_0)(\bar y-y) + \int_S\bar y-y d\mu \\
 &- \int_{S_h}(y_h-\hat y_0)(y_h-y^h) -\sum_{j=1}^m\mu_j(y_h-y^h)(x_j) \\
 &+ O(h^2)\|u_h-u\|_U \|{\check p}_h\|_{L^2(S)}.
\end{aligned}
\end{equation}
We rewrite the sum of the first and the third summand on the right-hand side as
\begin{equation}
\begin{aligned}
\int_S(y-&\tilde y_h+\tilde y_h-y_0)(\bar y- y)  - \int_{S_h}(y_h-\hat y_0)(y_h-y^h) \\
=& \int_S(y-\tilde y_h)(\bar y- \tilde y_h+\tilde y_h-y) +O(h^2)\|\tilde y_h-y_0\|_{L^2(S)}\|\bar y-y\|_{L^2(S)}\\
&+ \int_{S_h}(y_h-\hat y_0)(\hat{\bar y}-y_h+y^h-\hat y) \\
\le& -\|y-\tilde y_h\|_{L^2(S)}+\|y_h-\hat y_0\|_{L^2(S_h)}(\|\hat{\bar y}-y_h\|_{L^2(S_h)}+\|y^h-\hat y\|_{L^2(S_h)})\\
&+ O(h^2)\|\tilde y_h-y_0\|_{L^2(S)}\|u-u_h\|_U
.
\end{aligned}
\end{equation}
We use
\begin{equation}
\begin{aligned}
\bar y-y \le \bar y-\tilde y_h+\tilde y_h-\widetilde{I_hb}+\widetilde{I_hb}-b+b-y
\end{aligned}
\end{equation}
and
\begin{equation}
y_h \le I_hb \quad \wedge \quad \int_Sb-yd\mu =0
\end{equation}
so that
\begin{equation}
\int_S\bar y-y d\mu \le c \|\mu\|_{M(S)}\left\{\|\bar y-\tilde y_h\|_{L^{\infty}(S)}+\|\widetilde{I_hb}-b\|_{L^{\infty}(S)}\right\}.
\end{equation}
Using (\ref{c4}) and putting facts together leads to
\begin{equation}
\begin{aligned}
\|u-u_h\|^2_U +& \|y-\tilde y_h\|^2_{L^2(S)} \le \|u_0-u_{0,h}\|^2_U\\
&+\|\tilde y_h-y_0\|_{L^2(S)}(\|\hat {\bar y}-y_h\|_{L^2(S_h)}+\|y^h-\hat y\|_{L^2(S_h)}\\
&+O(h^2)\|u-u_h\|_U) \\
&+c(\|\bar y-\tilde y_h\|_{L^{\infty}(S)}+\|\widetilde{I_hb}-b\|_{L^{\infty}(S)}+\|y^h-\hat y\|_{L^{\infty}(S_h)}) \\
&+ + O(h^2)\|u_h-u\|_U \|{\check p}_h\|_{L^2(S)}
\end{aligned}
\end{equation}
Using Lemma \ref{c5}, Lemma 3.1 (iii) and (\ref{c6}) then yields
\begin{equation}
\|u-u_h\|^2_U + \|y-\tilde y_h\|^2_{L^2(S)} \le c(h^2+h^2|\log h|^2)
\end{equation}
so that the claim follows.
\end{proof}

\end{document}